\documentclass{article}%
\usepackage{amsmath}
\usepackage{amsfonts}
\usepackage{amssymb}
\usepackage{graphicx}%
\providecommand{\U}[1]{\protect\rule{.1in}{.1in}}
\providecommand{\U}[1]{\protect\rule{.1in}{.1in}}
\newtheorem{theorem}{Theorem}[section]

\newtheorem{lemma}[theorem]{Lemma}
\newtheorem{proposition}[theorem]{Proposition}
\newenvironment{proof}[1][Proof]{\noindent\textbf{#1.} }{\ \rule{0.5em}{0.5em}}
\begin{document}

\title{On normal subgroups of compact groups}
\author{Nikolay Nikolov and Dan Segal \\ Accepted for publication in J. European Math. Soc.}
\maketitle

\noindent N. Nikolov, Mathematical Institute, Oxford, OX1 6GG,
UK \newline nikolov@maths.ox.ac.uk \newline$\medskip$

\noindent D. Segal, All Souls College, Oxford OX1 4AL, UK \newline
dan.segal@all-souls.ox.ac.uk\bigskip

\begin{center}
\emph{Abstract}$\medskip$
\end{center}

\noindent Among compact Hausdorff groups $G$ whose maximal profinite quotient
is finitely generated, we characterize those that possess a proper dense
normal subgroup. We also prove that the abstract commutator subgroup $[H,G]$
is closed for every closed normal subgroup $H$ of $G$.

\newpage\setcounter{section}{-1}

\section{Introduction}

Let $G$ be a compact Hausdorff topological group with identity component
$G^{0}$, so $G^{0}$ is a connected compact group and $G/G^{0}$ is the maximal
profinite quotient of $G$. We say that $G$ is of \emph{f.g. type} if $G/G^{0}
$ is (topologically) finitely generated. In \cite{NS2} we established a number
of results about groups of this type, including

\begin{theorem}
\label{thm1}The (abstract) derived group $G^{\prime}=[G,G]$ is closed in $G$.
\end{theorem}

\begin{theorem}
\label{vdns}$G$ has a virtually-dense normal subgroup of infinite index if and
only if $G$ has an open normal subgroup $G_{1}$ such that \emph{either}
$G_{1}$ has infinite abelianization \emph{or} $G_{1}$ has a strictly infinite
semisimple quotient.
\end{theorem}

\noindent(A \emph{strictly infinite semisimple group} is a Cartesian product
of infinitely many finite simple groups or compact connected simple Lie
groups). The purpose of this note is (1) to generalize Theorem \ref{thm1}, and
(2) to characterize the compact groups of f.g. type that possess
a\emph{\ dense} proper normal subgroup (which must then have infinite index,
by the main result of \cite{NS1}); this analogue to Theorem \ref{vdns} is a
little technical to state: see Theorem \ref{dns} in Section \ref{dnssec} below.

Here and throughout, \emph{simple} group is taken to mean \emph{non-abelian}
simple group.

\section{Commutators}

\subsection{The main result}

\begin{theorem}
\label{[H,G]thm}Let $G$ be a compact Hausdorff group of f.g. type and $H$ a
closed normal subgroup of $G$. Then the (abstract) commutator subgroup $[H,G]
$ is closed in $G$.
\end{theorem}

When $G$ is profinite this is Theorem 1.4 of \cite{NS1}; when $G$ is
connected, it may be deduced quite easily from known results (cf. \cite{HM},
Theorem 9.24). The general case, however, seems harder: to prove it we have to
beef up slightly the main technical result of \cite{NS2}; the beefed-up
version appears as Proposition \ref{KT} below.

Now let $G$ and $H$ be as in Theorem \ref{[H,G]thm}. For a subset
$\mathcal{X}$ of $G$ we let $\overline{\mathcal{X}}$ denote the closure of
$\mathcal{X}$ in $G$, and for $f\in\mathbb{N}$ write%
\[
\mathcal{X}^{\ast f}=\left\{  x_{1}x_{2}\ldots x_{f}\mid x_{1},\ldots,x_{f}%
\in\mathcal{X}\right\}  .
\]
The minimal size of a finite topological generating set for $G$ (if there is
one) is denoted $\mathrm{d}(G)$.

$\medskip$

\emph{First reduction.}$\medskip$

Let $\mathcal{N}$ denote the set of closed normal subgroups $N$ of $G$ such
that $G/N$ is a Lie group. Set%
\begin{equation}
\mathcal{X}=\left\{  [h,g]\mid h\in H,~g\in G\right\}  . \label{X}%
\end{equation}
Then $\mathcal{X}$ is a compact set, so $\mathcal{X}^{\ast f}$ is closed for
each finite $f$. Suppose that for some $f$ we have%
\[
\lbrack H,G]N=\mathcal{X}^{\ast f}N~\forall N\in\mathcal{N}.
\]
Then%
\[
\lbrack H,G]\subseteq\bigcap_{N\in\mathcal{N}}\mathcal{X}^{\ast f}%
N=\overline{\mathcal{X}^{\ast f}}=\mathcal{X}^{\ast f}\subseteq\lbrack H,G]
\]
(cf. \cite{HM}, Lemma 9.1), so $[H,G]$ is closed.

Thus it will suffice to prove

\begin{theorem}
\label{T3}Let $G$ be a compact Lie group and $H$ a closed normal subgroup of
$G$. Then
\[
\lbrack H,G]=\mathcal{X}^{\ast f}%
\]
where $\mathcal{X}$ is given by (\ref{X}) and $f$ depends only on
$\mathrm{d}(G/G^{0})$.
\end{theorem}

\emph{Second reduction.}$\medskip$

We assume now that $G$ is a compact Lie group, with $\mathrm{d}(G/G^{0})=d$.
Then $G/G^{0}$ is finite and $G=G^{0}\Gamma$ for some finite subgroup
$\Gamma=\left\langle g_{1},\ldots,g_{2d}\right\rangle $ where $g_{2i}%
=g_{2i-1}^{-1}$ for each $i$; also $G^{0}=ZS$ where $Z=\mathrm{Z}(G^{0})$ and
$S$ is connected semisimple (cf. \cite{HM} Theorems 6.36, 6.15, 6.18).

Put $H_{1}=HZ\cap S$, $H_{2}=H_{1}^{0}$. Then $H_{2}$ is connected semisimple,
so every element of $H_{2}$ is a commutator (\cite{HM} Cor. 6.56); it follows
that $H_{2}\subseteq\mathcal{X}$. So replacing $G$ by $G/H_{2}$ we may suppose
that $H_{1}$ is finite, which implies that $H_{1}\leq Z$ since $H_1$ is normal in $S$ and $S$ is connected. Then $H_{3}:=H\cap
G^{0}\leq ZH_{1}=Z$.

As $H_{3}$ is abelian we have%
\[
\lbrack H_{3},G]=[H_{3},\Gamma]=\prod_{i=1}^{d}[H_{3},g_{2i}]\subseteq
\mathcal{X}^{\ast d},
\]
and $\mathcal{X}^{\ast d}$ is closed; so replacing $G$ by $G/\overline
{[H_{3},G]}$ we may suppose that $H_{3}\leq\mathrm{Z}(G)$.

Let $x\in H$ and $y\in G^{0}$. Then $x^{n}\in H_{3}$ where $n=\left\vert
G:G^{0}\right\vert $, and there exists $v\in G^{0}$ such that $y=v^{n}$ (every
element of $G^{0}$ is an $n$th power, because it lies in a torus: see
\cite{HM}, Theorem 9.32). Now%
\[
1=[x^{n},v]=[x,v]^{n}=[x,v^{n}]=[x,y]\text{.}%
\]
Thus $G^{0}\leq\mathrm{C}_{G}(H)$.

Put $D=H\Gamma\cap G^{0}$. Then $H\Gamma=D\Gamma$. Since $H\Gamma/H_{3}$ is
finite, we have $H_{3}\geq(H\Gamma)^{0}=H^{0}:=H_{4}$. There exists a finite
subgroup $L$ of $H\Gamma$ such that $H\Gamma=H_{4}L$ (\cite{HM}, Theorem
6.74). Then $L\vartriangleleft H\Gamma$ (since $H_4 \leq \mathrm{Z}(G)$) and $\Delta:=L\cap H\vartriangleleft
G^{0}\Gamma=G$ since $[H,G^0]=1$. Moreover $H=H_{4}\Delta$; and $[H,G]=[\Delta,G]$.

Applying \cite{HM}, Theorem 6.74 to the group $G/\Delta$, we find a finite
subgroup $Q/\Delta$ of $G/\Delta$ such that $G=G^{0}Q$ and $Q/\Delta \cap (G/\Delta)^0 \leq \mathrm{Z}((G/\Delta)^0)$, hence $Q\cap G^{0}%
\Delta\vartriangleleft G$. Then $Q\cap G^{0}$ is central in $G^{0}$. Replacing
each $g_{i}$ by an element of $G^{0}g_{i}\cap Q$, we may suppose that
$\Gamma\leq Q$. Now putting%
\[
E=\Delta\Gamma \leq Q, \ \textrm{and }A=E\cap G^{0} \leq \mathrm{Z}(G^0)
\]
we have%
\begin{align*}
E  &  =\Delta\Gamma=A\Gamma,\\
\lbrack A,A\Delta]  &  =[A\cap\Delta,E]=1,\\
\lbrack H,G]  &  =[\Delta,\Gamma]=[\Delta,E].
\end{align*}

\emph{Conclusion of the proof.}$\medskip$

In the following subsection we establish

\begin{theorem}
\label{newcomm}Let $G$ be a finite group, $H$ a normal subgroup of $G$ and
$\{y_{1},\ldots,y_{r}\}$ a symmetric subset of $G$. Suppose that
$G=H\left\langle y_{1},\ldots,y_{r}\right\rangle =A\left\langle y_{1}%
,\ldots,y_{r}\right\rangle $ where $A$ is an abelian normal subgroup of $G$
with $[A,H]=1$. Then%
\[
\lbrack H,G]=\left(
{\displaystyle\prod\limits_{i=1}^{r}}
[H,y_{i}]\right)  ^{\ast f_{1}}%
\]
where $f_{1}=f_{1}(r)$.
\end{theorem}

Applying this with $E,\Delta$ in place of $G,H$ we obtain%
\[
\lbrack H,G]=[\Delta,E]=\left(
{\displaystyle\prod\limits_{i=1}^{2d}}
[\Delta,g_{i}]\right)  ^{\ast f_{1}(2d)}.
\]
Taking account of all the reductions, we see that Theorem \ref{T3} follows,
with%
\[
f=1+d+2df_{1}(2d).
\]

\subsection{A variant of the `Key Theorem'}

Theorem \ref{newcomm} depends on the following proposition. We recall some
notation from \cite{NS2}; throughout this subsection $G$ will be a finite
group.$\medskip$

\noindent\textbf{Notation} For $\mathbf{g},\mathbf{v}\in G^{(m)}$ and $1\leq
j\leq m$,%
\[
\tau_{j}(\mathbf{g},\mathbf{v})=v_{j}[g_{j-1},v_{j-1}]\ldots\lbrack
g_{1},v_{1}],
\]

\[
\mathbf{v}\cdot\mathbf{g}=(v_{1}g_{1},\ldots,v_{m}g_{m}),~~~\mathbf{c}%
(\mathbf{v},\mathbf{g})=%
{\displaystyle\prod\limits_{j=1}^{m}}
[v_{j},g_{j}].
\]

\begin{proposition}
\label{KT}There exists a function $k:\mathbb{N}\rightarrow\mathbb{N}$ with the
following property. Let $G$ be a finite group, $H=[H,G]$ a soluble normal
subgroup of $G$, and $C\leq\mathrm{C}_{G}(H)$ a normal subgroup of $G$.
Suppose that $G=C\left\langle g_{1},\ldots,g_{r}\right\rangle $. Put $m=r\cdot
k(r)$, and for $1\leq j<k(r)$ and $1\leq i\leq r$ set%
\[
g_{i+jr}=g_{i}.
\]
Then for each $h\in H$ there exist $\mathbf{v}(i)\in H^{(m)}$ ($i=1,2,3$) such
that%
\begin{equation}
h=%
{\displaystyle\prod\limits_{i=1}^{3}}
\mathbf{c}(\mathbf{v}(i),\mathbf{g}) \label{hprod}%
\end{equation}
and%
\begin{equation}
C\left\langle g_{1}^{\tau_{1}(\mathbf{g},\mathbf{v}(i))},\ldots,g_{m}%
^{\tau_{m}(\mathbf{g},\mathbf{v}(i))}\right\rangle =G\qquad\text{for
}i=1,~2,~3. \label{G-gen}%
\end{equation}

\end{proposition}

In fact we can take%
\begin{equation}
k(r)=1+4(r+1)\cdot\max\{r,7\}. \label{k-def}%
\end{equation}

This reduces to (a special case of) Theorem 3.10 of \cite{NS2} when $C=1$. The
latter can be beefed up in a similar way in the general case where $H$ is not
necessarily soluble; as this will not be needed here, we leave it for the
interested reader to fill in the details. The ingredients of the proof (in
both cases) are all taken from \cite{NS2}, Section 3, though they need to be
arranged in a different way.

A normal subgroup $N$ of a group $G$ is said to be \emph{quasi-minimal normal}
if $N$ is minimal subject to%
\[
1<N=[N,G].
\]
Let $Z=Z_{N}$ be a normal subgroup of $G$ maximal subject to $Z<N$. Then
$[Z,_{n}G]=1$ for some $n$, which implies (i) that $Z=N\cap\zeta_{\omega}(G)$
is uniquely determined, and (ii) that $[Z,N]\leq\lbrack Z,H]\leq\lbrack
Z,G_{\omega}]=1$. (Here $\zeta_{\omega}(G)$ denotes the hypercentre and
$G_{\omega}$ the nilpotent residual of $G$; if $G$ is finite, $G_{\omega}$ is
the last term of the lower central series of $G$ and $[\zeta_{\omega
}(G),G_{\omega}]=1$.) An elementary argument (cf. \cite{NS2}, Lemma 3.4) shows
that $Z$ is contained in the Frattini subgroup $\Phi(G)$ of $G$. If $N$ is
soluble then $\overline{N}=N/Z$ is an abelian chief factor of $G$.

\bigskip

We fix $k=k(r)$ as given by (\ref{k-def}). Fix $h\in H$. For
$S\vartriangleleft G$ let us say that $\mathbf{v}=(v(i)_{j})\in H^{(3m)}$
satisfies E($S$), resp. G($S$) if (\ref{hprod}), resp. (\ref{G-gen}) is true
modulo $S$. By hypothesis, E($H$) and G($H$) are satisfied by $\mathbf{v}%
=(1,\ldots,1)$.

Since $H=[H,G]$ there is a chain%
\[
H=H_{0}>H_{1}\geq\ldots\geq H_{z}=1
\]
such that $H_{i-1}/H_{i}$ is a quasi-minimal normal subgroup of $G/H_{i}$ for
$i=1,\ldots,z$. Fix $l<z$ and suppose that $\mathbf{u}\in H^{(3m)}$ satisfies
E($H_{l}$) and G($H_{l}$). Our aim is to find elements $\mathbf{a}(i)\in
H_{l}{}^{(m)}$ such that $\mathbf{v}=\mathbf{a}\cdot\mathbf{u}$ satisfies
E($H_{l+1}$) and G($H_{l+1}$). If we can do this, the proposition will follow
by induction.

To simplify notation we now replace $G$ by $G/H_{l+1}$. Put $N=H_{l}$, now a
soluble quasi-minimal normal subgroup of $G$, and set $Z=Z_{N}$.

The argument is done in three steps. Put $K_{1}=N,$ $K_{2}=N^{\prime}[Z,G],$
$K_{3}=[Z,G]$, $K_{4}=1$. We are given that $\mathbf{u}$ satisfies E($N$) and
G($N$). Fix $q\leq3$ and suppose that $\mathbf{u}$ satisfies E($K_{q}$) and
G($K_{q}$); we will find $\mathbf{a}(i)\in N^{(m)}$ such that $\mathbf{v}%
=\mathbf{a}\cdot\mathbf{u}$ satisfies E($K_{q+1}$) and G($K_{q+1}$). Again, to
simplify notation we may replace $G$ by $G/K_{q+1}$ and so assume that
$K_{q+1}=1$, and set $K=K_{q}$. Thus we have to show that (\ref{hprod}) and
(\ref{G-gen}) hold.

\begin{lemma}
\label{u-v}%
\[
\left(
{\displaystyle\prod\limits_{i=1}^{3}}
\mathbf{c}(\mathbf{a}(i)\cdot\mathbf{u}(i),\mathbf{g)}\right)  \left(
{\displaystyle\prod\limits_{i=1}^{3}}
\mathbf{c}(\mathbf{u}(i),\mathbf{g})\right)  ^{-1}=%
{\displaystyle\prod\limits_{i=1}^{3}}
\left(
{\displaystyle\prod\limits_{j=1}^{m}}
[a(i)_{j},g_{j}]^{\tau_{j}(\mathbf{g},\mathbf{u}(i))}\right)  ^{w(i)}%
\]
where $w(i)=\mathbf{c(u}(i-1),\mathbf{g})^{-1}\ldots\mathbf{c}(\mathbf{u}%
(1),\mathbf{g})^{-1}$.
\end{lemma}

This is a direct calculation. The next lemma is easily verified by induction
on $m$ (see \cite{NS1}, Lemma 4.5):

\begin{lemma}
\label{newgens}%
\[
\left\langle g_{j}^{\tau_{j}(\mathbf{g},\mathbf{u})}\mid j=1,\ldots
,m\right\rangle =\left\langle g_{j}^{u_{j}h_{j}}\mid j=1,\ldots
,m\right\rangle
\]
where $h_{j}=g_{j-1}^{-1}\ldots g_{1}^{-1}$.
\end{lemma}

Now we are given $\mathbf{u}(i)\in H^{(m)}$ and $\kappa\in K$ such that%
\[
h=\kappa%
{\displaystyle\prod\limits_{i=1}^{3}}
\mathbf{c}(\mathbf{u}(i),\mathbf{g})
\]
and%
\begin{align}
G  &  =CK\left\langle g_{j}^{\tau_{j}(\mathbf{g},\mathbf{u}(i))}\mid
j=1,\ldots,m\right\rangle \nonumber\\
&  =CK\left\langle g_{j}^{u(i)_{j}h_{j}}\mid j=1,\ldots,m\right\rangle
\qquad\text{for }i=1,2,3 \label{u-genmodK}%
\end{align}
the second equality thanks to Lemma \ref{newgens}.

Let $\mathbf{v}=\mathbf{a}\cdot\mathbf{u}$ with $\mathbf{a}(i)\in N^{(m)}$ and the goal is to find a suitable $\mathbf{a}$.
Lemma \ref{u-v} shows that (\ref{hprod}) is then equivalent to%
\begin{equation}%
{\displaystyle\prod\limits_{i=1}^{3}}
\left(
{\displaystyle\prod\limits_{j=1}^{m}}
[a(i)_{j},g_{j}]^{\tau_{j}(\mathbf{g},\mathbf{u}(i))}\right)  ^{w(i)}%
=\kappa\text{.} \label{A-prod}%
\end{equation}
This can be further simplified by setting%
\begin{align}
y(i)_{j}  &  =g_{j}^{\tau_{j}(\mathbf{g},\mathbf{u}(i))w(i)},~~t(i)_{j}%
=g_{j}^{u(i)_{j}h_{j}}\nonumber\\
b(i)_{j}  &  =a(i)_{j}^{\tau_{j}(\mathbf{g},\mathbf{u}(i))w(i)},~~c(i)_{j}%
=a(i)_{j}^{u(i)_{j}h_{j}}. \label{a-b-c}%
\end{align}

Define $\phi(i):N^{(m)}\rightarrow N$ by%
\begin{equation}
\mathbf{b}\phi(i)=\mathbf{c}(\mathbf{b},\mathbf{y}(i)), \quad \mathbf{b} \in N^{(m)}. \label{phi}%
\end{equation}
Then (\ref{A-prod}) becomes%
\begin{equation}%
{\displaystyle\prod\limits_{i=1}^{3}}
\mathbf{b}(i)\phi(i)=\kappa, \label{kappa}%
\end{equation}
and (\ref{u-genmodK}) is equivalent to%
\begin{align}
G  &  =CK\left\langle y(i)_{1},\ldots,y(i)_{m}\right\rangle \nonumber\\
&  =CK\left\langle t(i)_{1},\ldots,t(i)_{m}\right\rangle \qquad\text{for
}i=1,2,3. \label{y-t-gen}%
\end{align}

Similarly, using Lemma \ref{u-v} again (\ref{G-gen}) holds if and only if for $i=1,~2,~3$ we have%
\begin{equation}
G=CZ\left\langle t(i)_{j}^{c(i)_{j}}\mid j=1,\ldots,m\right\rangle
\label{t-c-gen}%
\end{equation}
(where $Z$ is added harmlessly since $Z\leq\Phi(G)$).

Let $\mathcal{X}(i)$ denote the set of all $\mathbf{c}(i)\in N^{(m)}$ such
that (\ref{t-c-gen}) holds, and write $W(i)$ for the image of $\mathcal{X}(i)
$ under the bijection $N^{(m)}\rightarrow N^{(m)}$ defined in (\ref{a-b-c})
sending $\mathbf{c}(i)\longmapsto\mathbf{b}(i)$.

To sum up: to establish the existence of $\mathbf{a}(1),\mathbf{a}%
(2),\mathbf{a}(3)\in N^{(m)}$ such that the $\mathbf{v}(i)=\mathbf{a}%
(i)\cdot\mathbf{u}(i)$ satisfy (\ref{hprod}) and (\ref{G-gen}), it suffices to
find $(\mathbf{b}(1),\mathbf{b}(2),\mathbf{b}(3))\in W(1)\times W(2)\times
W(3)$ such that (\ref{kappa}) holds.

$\medskip$

We set $\varepsilon=\min\{\frac{1}{7},\frac{1}{r}\}$, and will write
$^{-}:G\rightarrow G/Z$ for the quotient map.

$\medskip$

\emph{The case }$q=1$

$\medskip$

In this case we have $K=N$ and we are assuming that $K_{2}=N^{\prime}[Z,G]=1$.
We use additive notation for $N$ and consider it as a $G$-module. Note that
$[CK,N]=1$. Then (\ref{y-t-gen}) together with $N=[N,G]$ imply that%
\[
\phi(1):\mathbf{b\longmapsto}%
{\textstyle\sum\nolimits_{j=1}^{m}}
b_{j}(y(1)_{j}-1)
\]
is a surjective ($\mathbb{Z}$-module) homomorphism $N^{(m)}\rightarrow N$. It
follows that%
\[
\left\vert \phi(1)^{-1}(c)\right\vert =\left\vert \ker\phi(1)\right\vert
=\left\vert N\right\vert ^{m-1}%
\]
for each $c\in N$.

Now fix $i\in\{1,2,3\}$. According to Theorem 2.1 of \cite{NS2}, at least one
of the elements $g_{j}$ has the $\varepsilon/2$-fixed-point space property on
$\overline{N}$ (see \cite{NS2}, section 2.1); therefore at least $k$ of the
elements $\overline{t(i)_{j}}$ have this property. Now we apply \cite{NS2},
Proposition 2.8(i) to the group $G/CZ$: if $N\nleq CZ$ this shows that
(\ref{t-c-gen}) holds for at least $\left\vert \overline{N}\right\vert
^{m}(1-\left\vert \overline{N}\right\vert ^{r-k\varepsilon/2})$ values of
$\overline{\mathbf{c}(i)}$ in $\left\vert \overline{N}\right\vert ^{m}$. If
$N\leq CZ$ the same holds trivially for \emph{all} $\overline{\mathbf{c}(i)}$
in $\left\vert \overline{N}\right\vert ^{m}$. It follows in any case that
\begin{equation}
\left\vert W(i)\right\vert =\left\vert \mathcal{X}(i)\right\vert
\geq\left\vert Z\right\vert ^{m}\cdot\left\vert \overline{N}\right\vert
^{m}(1-\left\vert \overline{N}\right\vert ^{r-k\varepsilon/2})=\left\vert
N\right\vert ^{m}(1-\left\vert \overline{N}\right\vert ^{r-k\varepsilon/2}).
\label{W(i)-sol}%
\end{equation}

We need to compare $\left\vert \overline{N}\right\vert $ with $\left\vert
N\right\vert $. Observe that $\mathbf{b\longmapsto}%
{\textstyle\sum\nolimits_{j=1}^{r}}
b_{j}(g_{j}-1)$ induces an epimorphism from $\overline{N}^{(r)}$ onto $N$;
consequently $\left\vert N\right\vert \leq\left\vert \overline{N}\right\vert
^{r}$. Thus since $k\varepsilon/2r>1$ we have%
\[
\left\vert W(i)\right\vert \geq\left\vert N\right\vert ^{m}(1-\left\vert
N\right\vert ^{1-k\varepsilon/2r})>0,
\]
so $W(i)$ is non-empty for each $i$. For $i=2,~3$ choose $\mathbf{b}(i)\in
W(i)$ and put%
\[
c=\kappa\left(
{\displaystyle\prod\limits_{i=2}^{3}}
\mathbf{b}(i)\phi(i)\right)  ^{-1}.
\]
Then%
\[
\left\vert \phi(1)^{-1}(c)\right\vert +\left\vert W(1)\right\vert
\geq\left\vert N\right\vert ^{m}(\left\vert N\right\vert ^{-1}+1-\left\vert
N\right\vert ^{1-k\varepsilon/2r})>\left\vert N\right\vert ^{m}%
\]
since $k\varepsilon/2r>2$. It follows that $\phi(1)^{-1}(c)\cap W(1)$ is
non-empty. Thus we may choose $\mathbf{b}(1)\in\phi(1)^{-1}(c)\cap W(1)$ and
ensure that (\ref{kappa}) is satisfied.

$\medskip$

\emph{The case }$q=2$

$\medskip$

Now we take $K=N^{\prime}$, assuming that $K_{3}=[Z,G]=1$. Since $N^{\prime
}\leq Z$, the argument above again gives (\ref{W(i)-sol}).

The maps $\phi(i)$ are no longer homomorphisms, however. Below we establish

\begin{proposition}
\label{NS6.2}Let $N$ be a soluble quasi-minimal normal subgroup and $C$ a
normal subgroup of the finite group $G$, with $[C,N]=1$. Assume that
$G=C\left\langle y(i)_{1},\ldots,y(i)_{m}\right\rangle $ for $i=1,2,3$. Then
for each $c\in N^{\prime}$ there exist $c_{1},c_{2},c_{3}\in N$ such that
$c=c_{1}c_{2}c_{3}$ and%
\begin{equation}
\left\vert \phi(i)^{-1}(c_{i})\right\vert \geq\left\vert N\right\vert
^{m}\cdot\left\vert \overline{N}\right\vert ^{-r-2}\qquad(i=1,2,3),
\label{fibres-sol}%
\end{equation}
where $\phi(i)$ is given by (\ref{phi}) and $r=\mathrm{d}(G/C)$.
\end{proposition}

Since now $K\leq Z$, the hypotheses of Proposition \ref{NS6.2} follow from (\ref{y-t-gen}). Put
$c=\kappa$ and choose $c_{1},c_{2},c_{3}$ as in the Proposition. As
$k\varepsilon>4r+4$, we see that (\ref{W(i)-sol}) and (\ref{fibres-sol})
together imply that $\phi(i)^{-1}(c_{i})\cap W(i)$ is non-empty for $i=1,2,3$.
Thus we can find $\mathbf{b}(i)\in\phi(i)^{-1}(c_{i})\cap W(i)$ for $i=1,2,3$
to obtain (\ref{kappa}).

$\medskip$

\emph{The case }$q=3$

$\medskip$

Now we take $K=[Z,G]$. Since $G=C\left\langle g_{1},\ldots,g_{r}\right\rangle
$, we have $K=%
{\textstyle\prod\nolimits_{j=1}^{r}}
[Z,g_{j}]$. Thus $\kappa=%
{\textstyle\prod\nolimits_{j=1}^{r}}
[z_{j},g_{j}]$ with $z_{1},\ldots,z_{r}\in Z$. In this case, (\ref{kappa}) is
satisfied if we set%
\begin{align*}
b(1)_{j}  &  =z_{j}\qquad(1\leq j\leq r)\\
b(1)_{j}  &  =1\qquad(r<j\leq m)\\
b(i)_{j}  &  =1\qquad(i=2,~3,~1\leq j\leq m),
\end{align*}
because $y(i)_{j}$ is conjugate to $g_{j}$ under the action of $H$ and
$[Z,H]=1$.

For each $i$ we have $W(i)\supseteq Z^{(m)}$, since in this case
(\ref{y-t-gen}) implies (\ref{t-c-gen}) if $c(i)_{j}\in Z$ for all $j$. So
$\mathbf{b}(i)\in W(i)$ for each $i$, as required.

\bigskip

This concludes the proof of Proposition \ref{KT}, modulo the

\bigskip

\textbf{Proof of Proposition \ref{NS6.2}.}$\medskip$

Now $N$ is a quasi-minimal normal subgroup of $G=C\left\langle g_{1}%
,\ldots,g_{r}\right\rangle $. Recall the definition of $Z_N$ as a normal subgroup of $G$ maximal subject to $Z <N$; we saw that $Z_N$ is in fact uniquely determined. Put $\Gamma=N\left\langle g_{1},\ldots
,g_{r}\right\rangle $. Let $g_{0}\in N\smallsetminus
Z_{N}$. Since $[C,N]=1$ we have $N=\langle g_0^G \rangle = \langle g_0^{\langle g_1, \ldots, g_r\rangle } \rangle$ and so $\Gamma=\left\langle
g_{1},\ldots,g_{r},g_{0}\right\rangle $; thus $d(\Gamma)\leq r+1$. For each $i$ and $j$ we have $y(i)_{j}%
=c_{ij}x_{ij}$ with $c_{ij}\in C$ and $x_{ij}\in\Gamma$. Then for $i=1,2,3$ we
have
\begin{align*}
\Gamma &  =(\Gamma\cap C)\left\langle x_{i1},\ldots,x_{im}\right\rangle \\
&  =\left\langle x_{i1},\ldots,x_{im},x_{i,m+1},\ldots,x_{im^{\prime}%
}\right\rangle
\end{align*}
where $m^{\prime}=m+r+1$ and $x_{ij}\in C$ for $m<j\leq m^{\prime}$.

Define $\psi(i):N^{(m^{\prime})}\rightarrow N$ by%
\[
\mathbf{b}\psi(i)=%
{\displaystyle\prod\limits_{j=1}^{m^{\prime}}}
[b_{j},x_{ij}].
\]
We apply Proposition 7.1 of \cite{NS1} to the group $\Gamma$ and its soluble
quasi-minimal normal subgroup $N$. This shows that for each $c\in N^{\prime}$
there exist $c_{1},c_{2},c_{3}\in N$ such that $c=c_{1}c_{2}c_{3}$ and%
\[
\left\vert \psi(i)^{-1}(c_{i})\right\vert \geq\left\vert N\right\vert
^{m^{\prime}}\cdot\left\vert \overline{N}\right\vert ^{-r-2}\qquad(i=1,2,3).
\]

Now%
\[
(b_{1},\ldots,b_{m^{\prime}})\psi(i)=(b_{1},\ldots,b_{m})\phi(i)
\]
for each $\mathbf{b}\in N^{(m^{\prime})}$; so%
\[
\psi(i)^{-1}(c_{i})=\phi(i)^{-1}(c_{i})\times N^{(r+1)}%
\]
and it follows that%
\[
\left\vert \phi(i)^{-1}(c_{i})\right\vert =\left\vert \psi(i)^{-1}%
(c_{i})\right\vert \cdot\left\vert N\right\vert ^{-(r+1)}\geq\left\vert
N\right\vert ^{m}\cdot\left\vert \overline{N}\right\vert ^{-r-2}%
\]
as required.

This completes the proof.

\bigskip

\textbf{Proof of Theorem \ref{newcomm}. }Now $G$ is a finite group, $H$ is a
normal subgroup of $G$ and $\{y_{1},\ldots,y_{r}\}$ is a symmetric subset of
$G$. We are given that $G=H\left\langle y_{1},\ldots,y_{r}\right\rangle
=A\left\langle y_{1},\ldots,y_{r}\right\rangle $ where $A$ is an abelian
normal subgroup of $G$ with $[A,H]=1$. The claim is that%
\[
\lbrack H,G]=\left(
{\displaystyle\prod\limits_{i=1}^{r}}
[H,y_{i}]\right)  ^{\ast f_{1}}%
\]
where $f_{1}=f_{1}(r)$.

Put $\Gamma=\left\langle y_{1},\ldots,y_{r}\right\rangle $, so $G=H\Gamma
=A\Gamma$. Choose $n$ so that $[H,_{n}G]=K$ satisfies $K=[K,G]$. By \cite{W},
Proposition 1.2.5 we have%
\[
\lbrack H,G]=%
{\displaystyle\prod\limits_{i=1}^{r}}
[H,y_{i}]\cdot K.
\]
Put $G_{1}=K\Gamma$ and $A_{1}=A\cap G_{1}$. Then $G_{1}=A_{1}\Gamma$ and
$K=[K,\Gamma]=[K,G_{1}]$. So replacing $H$ by $K,$ $G$ by $G_{1}$ and $A$ by
$A_{1}$ we reduce to the case where $H=[H,G]$. This implies in particular that%
\[
G=H\Gamma=G^{\prime}\Gamma=G^{\prime}\left\langle y_{1},\ldots,y_{r}%
\right\rangle .
\]

Now $AH\cap\Gamma$ is centralized by $A$ and normalized by $\Gamma$, so
$AH\cap\Gamma\vartriangleleft G$. But $AH=A(AH\cap\Gamma)$ so%
\[
H^{\prime}=(AH)^{\prime}=(AH\cap\Gamma)^{\prime}\leq\Gamma\text{.}%
\]
Theorem 1.2 of \cite{NS2} gives%
\[
\lbrack H^{\prime},G]=[H^{\prime},\Gamma]=\left(  \prod_{i=1}^{r}[H^{\prime
},y_{i}]\right)  ^{\ast f_{0}(r)}%
\]
(where $f_{0}(r)$ depends only on $r$). Replacing $G$ by $G/[H^{\prime},G]$,
we reduce to the case where $H^{\prime}\leq\mathrm{Z}(G)$. Thus $H=[H,G]$ is
nilpotent. It follows by Proposition \ref{KT} that%
\[
H=\left(  \prod_{i=1}^{r}[H,y_{i}]\right)  ^{\ast3k(r)}.
\]
Putting everything together we can take%
\[
f_{1}=1+f_{0}(r)+3k(r).
\]

(The alert reader may wonder how we could establish the hard result Theorem
\ref{newcomm} using only a version of the easier, `soluble' case of the `Key
Theorem' from \cite{NS2}; the answer is that the full strength of the latter
is implicitly invoked at the point where we quote \cite{NS2}, Theorem 1.2.)

\section{Dense normal subgroups\label{dnssec}}

\subsection{The main result}

\noindent\textbf{Definition.} \textbf{(a) }Let $S$ be a finite simple group.
Then $Q(S)$ denotes the following subgroup of $\mathrm{Aut}(S)$:%
\begin{align*}
&  \mathrm{PGO}_{2n}^{+}(q)\text{ if }S=D_{n}(q),~n\geq5\\
&  \mathrm{PGO}_{2n}^{-}(q)\text{ if }S=\,^{2}\!D_{n}(q)\\
&  \mathrm{Inn}(S)\text{ if }S=C_{n}(q)\\
&  \mathrm{InnDiag}(S)\text{ if }S\text{ is of another Lie type}\\
&  \mathrm{Aut}(S)\text{ in all other cases}%
\end{align*}

Note that when $S=D_{n}(q)$ then $Q(S)=\mathrm{PGO}_{2n}^{+}(q)=\mathrm{PSO}%
_{2n}^{+}(q)\langle\tau\rangle$ and when $S=\,^{2}\!D_{n}(q)$ then
$Q(S)=\mathrm{PGO}_{2n}^{+}(q)=\mathrm{PSO}_{2n}^{+}(q)\langle\lbrack q]\rangle$,
where $\tau$ is the non-trivial graph automorphism of $D_{n}(q)$ and $[q]$
denotes the field automorphism of order $2$ of $^{2}\!D_{n}(q)$.

\textbf{(b) }Let $S$ be a connected compact simple Lie group. Then%
\[
Q(S)=\left\{
\begin{array}
[c]{ccc}%
\mathrm{Aut}(S) & \text{if} & S=\mathrm{PSO}(2n),~n\geq3\\
&  & \\
\mathrm{Inn}(S) &  & \text{else}%
\end{array}
\right.  .
\]

\textbf{(c) }A compact topological group $H$ is \emph{Q-almost-simple} if
$S\vartriangleleft H\leq Q(S)$ where $S$ is a finite simple group or a compact
connected simple Lie group with trivial centre (and $S$ is identified with
$\mathrm{Inn}(S)$ ). Note that if $H$ is not finite, then $H$ is
Q-almost-simple if and only if it is either simple or else isomorphic to
$\mathrm{Aut}(\mathrm{PSO}(2n))$ for some $n\geq3$, because $\left\vert
\mathrm{Aut}(S)/\mathrm{Inn}(S)\right\vert =2$ for $S=\mathrm{PSO}%
(2n)$.\bigskip

\noindent If $H$ is Q-almost-simple as above, the \emph{rank} of $H$ is
defined to be the (untwisted) Lie rank of $S$ if $S$ is of Lie type, $n$ if
$S\cong\mathrm{Alt}(n)$, and zero otherwise.

\begin{theorem}
\label{dns}Let $G$ be a compact Hausdorff group of f.g. type. Then $G$ has a
proper dense normal subgroup if and only if one of the following holds:

\begin{itemize}
\item $G$ has an infinite abelian quotient, \emph{or}

\item $G$ has a strictly infinite semisimple quotient, \emph{or}

\item $G$ has \emph{Q}-almost-simple quotients of unbounded ranks.
\end{itemize}
\end{theorem}

\noindent(The quotients here refer to $G$ as a \emph{topological} group, i.e.
they are continuous quotients -- in the first case this makes no difference,
in view of Theorem \ref{thm1}.)

\subsection{The profinite case\label{profsec}}

Let $G$ be an infinite finitely generated profinite group. It is clear that in
each of the following cases, $G$ has a \emph{countable}, hence proper, dense
normal subgroup:

\begin{itemize}
\item $G$ is abelian (because $G$ contains a dense (abstractly) finitely
generated subgroup)

\item $G$ is semisimple ($G$ is the Cartesian product of infinitely many
finite simple groups, and the restricted direct product is a dense normal subgroup).
\end{itemize}

Let $G_{2}$ denote the intersection of all maximal open normal subgroups of
$G$ not containing $G^{\prime}$; thus%
\[
G_{\mathrm{ss}}:=G/G_{2}%
\]
is the maximal semisimple quotient of $G$. The preceding observations imply:

\begin{itemize}
\item $G$ has a proper dense normal subgroup if either $G^{\mathrm{ab}%
}:=G/G^{\prime}$ or $G_{\mathrm{ss}}$ is infinite
\end{itemize}

\noindent(recall that $G^{\prime}$ is closed, so $G/G^{\prime}$ is again
profinite, by Theorem \ref{thm1}).

We recall a definition and a result from \cite{NS2}, Section 1:$\medskip$

\textbf{Definition.} $G_{0}$ denotes the intersection of the centralizers of
all simple non-abelian chief factors of $G\medskip$

\noindent(here by `chief factor' of $G$ we mean a chief factor of some $G/K$
where $K$ is an open normal subgroup of $G$.)

\begin{proposition}
\label{G_0}\emph{(\cite{NS2}, Corollary 1.8)} Let $N$ be a normal subgroup of
(the underlying abstract group) $G$. If $NG^{\prime}=NG_{0}=G$ then $N=G$.
\end{proposition}

Now let $\mathcal{X}$ denote the class of all finitely generated profinite
groups $H$ such that $H_{0}=1$ and both $H^{\mathrm{ab}}$ and $H_{\mathrm{ss}%
}$ are finite, and let $\mathcal{X}(\mathrm{dns})$ denote the subclass
consisting of those groups that contain a proper dense normal subgroup.

\begin{lemma}
\label{3-conds}Let $G$ be a finitely generated profinite group. Then $G$
contains a proper dense normal subgroup if and only if at least one of the
following holds:

\begin{description}
\item[(a)] $G^{\mathrm{ab}}$ is infinite

\item[(b)] $G_{\mathrm{ss}}$ is infinite

\item[(c)] $G/G_{0}\in\mathcal{X}(\mathrm{dns})$.
\end{description}
\end{lemma}

\begin{proof}
We have shown above that $G$ contains a proper dense normal subgroup if either
(a) or (b) holds, and the same clearly follows in case (c). Suppose conversely
that none of (a), (b) or (c) holds. Then $G/G_{0}\in\mathcal{X}\smallsetminus
\mathcal{X}(\mathrm{dns})$. Now let $N$ be a dense normal subgroup of $G$.
SInce $G^{\mathrm{ab}}$ is finite, $G^{\prime}$ is open in $G$ and so
$NG^{\prime}=G$. As $G/G_{0}$ has no proper dense normal subgroup, we also
have $NG_{0}=G$. Now Proposition \ref{G_0} shows that $N=G$.
\end{proof}

\bigskip

Thus it remains to identify the groups in $\mathcal{X}(\mathrm{dns})$.

$\medskip$

For any chief factor $S$ of $G$ let%
\[
\mathrm{Aut}_{G}(S)
\]
denote the image of $G$ in $\mathrm{Aut}(S)$, where $G$ acts by conjugation.
Now we can state

\begin{proposition}
\label{X(dns)}Let $G\in\mathcal{X}$. Then $G\in\mathcal{X}(\mathrm{dns})$ if
and only if the simple chief factors $S$ of $G$ such that%
\begin{equation}
\mathrm{Aut}_{G}(S)\leq Q(S) \label{Aut<Q(S)}%
\end{equation}
have unbounded ranks.
\end{proposition}

Since {$\mathrm{Aut}_{G}(S)$ is a Q-almost-simple image of $G$ for such chief
factor $S$, this will complete the proof of Theorem \ref{dns} in the case of a
profinite group $G$. }

$\medskip$

Proposition \ref{X(dns)} depends on the following four lemmas, which will be
sketched in the next subsection:

\begin{lemma}
\label{notinQ(S)}There exists $\varepsilon_{\ast}>0$ such that%
\[
\log\left\vert [S,f]\right\vert \geq\varepsilon_{\ast}\log\left\vert
S\right\vert
\]
whenever $S$ is a finite simple group and $f\in\mathrm{Aut}(S)\smallsetminus
Q(S)$.
\end{lemma}

\begin{lemma}
\label{bdedranklemma}If $S$ is a finite simple group of rank at most $r$ and
$1\neq f\in\mathrm{Aut}(S)$ then%
\[
\log\left\vert [S,f]\right\vert \geq\varepsilon(r)\log\left\vert S\right\vert
,
\]
where $\varepsilon(r)>0$ depends only on $r$.
\end{lemma}

\begin{lemma}
\label{epsilonlemma}Given $\varepsilon>0$, there exists $k(\varepsilon
)\in\mathbb{N}$ with the following property: if $S$ is a finite simple group
and $f\in\mathrm{Aut}(S)$ satisfies $\log\left\vert [S,f]\right\vert
\geq\varepsilon\log\left\vert S\right\vert $ then%
\[
S=\left(  [S,f][S,f^{-1}]\right)  ^{\ast k(\varepsilon)}.
\]

\end{lemma}

\begin{lemma}
\label{ranklemma}For every $\varepsilon>0$ there exists $n\in\mathbb{N}$ such
that if $S$ is a finite simple group of rank at least $n$ and $f\in Q(S) $,
then there exists $s\in S$ such that%
\[
\log\left\vert [S,sf]\right\vert <\varepsilon\log\left\vert S\right\vert .
\]

\end{lemma}

\noindent\textbf{Proof of Proposition \ref{X(dns)}.} Now $G^{\mathrm{ab}}$ and
$G_{\mathrm{ss}}$ are finite, and $G_{0}=1$. This implies that $G$ has a
semisimple closed normal subgroup $T=\overline{G^{(3)}}$, and the simple
factors of $T$ are precisely the simple chief factors of $G$ (here
$\overline{G^{(3)}}$ denotes the closure of the third term of the derived
series of $G$; see \cite{NS2}, Section 1.1). Thus $T=T_{0}\times T_{1}\times
T_{2}$ where $T_{0}$ is the product of those simple factors $S$ such that
$G=S\mathrm{C}_{G}(S)$, $T_{1}$ is the product of those $S\nleqslant T_{0}$
for which (\ref{Aut<Q(S)}) holds, and $T_{2}$ is the product of the rest. Note
that $T_{0}$ is finite, because $T_{0}\cong G_{\mathrm{ss}}$. We fix a finite
set $\{a_{1},\ldots,a_{d}\}$ of (topological) generators for $G$.

Suppose that $T_{1}$ contains factors of unbounded ranks. Pick a sequence
$(S_{j})$ of simple factors of $T_{1}$ such that $\mathrm{rank}(S_{j}%
)\rightarrow\infty$ and let $T_{3}=\prod_{j\in\mathbb{N}}S_{j}$. Then
$T=T_{3}\times T_{4}$ for a suitable complement $T_{4}$. If $G/T_{4}$ has a
proper dense normal subgroup then so does $G$; so replacing $G$ by $G/T_{4}$
we may assume that $T=T_{3}$.

In view of Lemma \ref{ranklemma}, we can find $s_{ij}\in S_{j}$ such that%
\[
\frac{\log\left\vert [S_{j},s_{ij}a_{i}]\right\vert }{\log\left\vert
S_{j}\right\vert }<\varepsilon_{j}%
\]
for all $i$ and $j$, where $\varepsilon_{j}\rightarrow0$ as $j\rightarrow
\infty$. For each \thinspace$i$ put%
\[
b_{i}=(s_{ij})\cdot a_{i}\in Ta_{i}\text{,}%
\]
and note that $[S_{j},s_{ij}a_{i}]=[S_{j},b_{i}]$. Let $N=\left\langle
b_{1}^{G},\ldots,b_{d}^{G}\right\rangle $ be the normal subgroup of $G$
generated (algebraically) by $b_{1},\ldots,b_{d}$.

Since $\mathrm{Aut}_{G}(S_{j})\neq\mathrm{Inn}(S_{j})$, for each $j$ there
exists $i$ such that $[S_{j},b_{i}]\neq1$, and so $1\neq\lbrack S_{j},N]\leq
S_{j}\cap N$. As each $S_{j}$ is simple it follows that $N$ contains the
(restricted) direct product $P=\left\langle S_{j}\mid j\in\mathbb{N}%
\right\rangle $. Therefore $\overline{N}\geq\overline{P}=T$, and as $a_{i}\in
TN$ for each \thinspace$i$ it follows that $N$ is dense in $G$.

On the other hand, $N\neq G$. To see this, fix $i$ and $j$, set $b=b_{i}$,
$S=S_{j}$, and write $\dag=\dag_{j}:G\rightarrow\mathrm{Aut}_{G}(S)$ for the
natural map. Then $\left\vert G^{\dag}\right\vert \leq\left\vert
\mathrm{Aut}(S)\right\vert \leq\left\vert S\right\vert ^{1+\eta_{j}}$ where
$\eta_{j}\rightarrow0$ as $j\rightarrow\infty$, because $\mathrm{rank}%
(S_{j})\rightarrow\infty$ (see \cite{GLS}, Section 2.5). So%
\[
\left\vert (b^{G})^{\dag}\right\vert =\left\vert [G^{\dag},b^{\dag
}]\right\vert \leq  |G^{\dag}:S^{\dag}| |[S^{\dag},b^{\dag}]| \leq \left\vert S\right\vert ^{\eta_{j}}\cdot\left\vert \lbrack
S,b]\right\vert \leq \left\vert S\right\vert ^{\eta_{j}+\varepsilon_{j}%
}\text{.}%
\]
Now for $n\in\mathbb{N}$ set $X_{n}=\left(  b_{1}^{G}\cup(b_{1}^{-1})^{G}%
\cup\ldots\cup b_{d}^{G}\cup(b_{d}^{-1})^{G}\right)  ^{\ast n}$. Then%
\begin{align*}
\left\vert X_{n}^{\dag}\right\vert  &  \leq(2d\left\vert S\right\vert
^{\eta_{j}+\varepsilon_{j}})^{n}\\
&  <\left\vert S_{j}\right\vert =\left\vert S_{j}^{\dag}\right\vert
\end{align*}
if $\left\vert S_{j}\right\vert >(2d)^{2n}$ and $\eta_{j}+\varepsilon
_{j}<(2n)^{-1}$. This holds for all sufficiently large values of $j$; thus we
may choose a strictly increasing sequence $(j(n))$ and for each $n$ an element
$x_{j(n)}\in S_{j(n)}$ such that%
\[
x_{j(n)}^{\dag_{j(n)}}\notin X_{n}^{\dag_{j(n)}}.
\]
Let $t\in T_{1}$ have $x_{j(n)}$ as its $S_{j(n)}$-component for each $n$.
Then $t^{\dag_{j(n)}}=x_{j(n)}^{\dag_{j(n)}}\notin X_{n}^{\dag_{j(n)}}$ for
every $n$, and so%
\[
t\notin\bigcup_{n=1}^{\infty}X_{n}=N;
\]
hence $N\neq G$ as claimed.

$\medskip$

For the converse, suppose that every simple factor of $T_{1}$ has rank at most
$r$. Let $N$ be a dense normal subgroup of $G$. Then $NT=G$ by Lemma
\ref{3-conds}, since $(G/T)_{0}=G/T$.

Let $a_{1},\ldots,a_{d}$ be as above and choose $c_{i}\in a_{i}T\cap N$. For
each simple factor $S$ of $T_{2}$, there exists $i$ such that conjugation by
$c_{i}$ induces on $S$ an automorphism not in $Q(S)$. Then Lemmas
\ref{epsilonlemma} and \ref{notinQ(S)} show that%
\[
S=\left(  [S,c_{i}][S,c_{i}^{-1}]\right)  ^{\ast k(\varepsilon_{\ast})}.
\]
It follows that%
\[
T_{2}=\left(  \prod_{i=1}^{d}[T,c_{i}][T,c_{i}^{-1}]\right)  ^{\ast
k(\varepsilon_{\ast})}\subseteq N\text{.}%
\]

For each simple factor $S$ of $T_{1}$, there exists $i$ such that conjugation
by $c_{i}$ does not centralize $S$. Since each such $S$ has rank at most $r$,
we see in the same way, now using Lemmas \ref{epsilonlemma} and
\ref{bdedranklemma}, that%
\[
T_{1}=\left(  \prod_{i=1}^{d}[T,c_{i}][T,c_{i}^{-1}]\right)  ^{\ast
k(\varepsilon(r))}\subseteq N.
\]

We conclude that%
\[
\left\vert G:N\right\vert =\left\vert T:T\cap N\right\vert \leq\left\vert
T:T_{1}T_{2}\right\vert =\left\vert T_{0}\right\vert <\infty.
\]
Therefore $N$ is open in $G$ by \cite{NS1}, Theorem 1.1 (=\cite{NS2}, Theorem
5.1), and so $N=G$.

\subsection{Some lemmas}

\noindent\textbf{Proof of Lemma \ref{epsilonlemma}. }$S$ is a simple group and
$f\in\mathrm{Aut}(S)$ satisfies $\log\left\vert [S,f]\right\vert
\geq\varepsilon\log\left\vert S\right\vert $. Put $Y=[S,f]^{S}$ and
$X=[S,f][S,f^{-1}].$ Then $Y^{\ast k(\varepsilon)}=S$ by Proposition 1.23 (a
result from \cite{LiSh}), where $k(\varepsilon)=\left\lceil c^{\prime
}/\varepsilon\right\rceil $; and $X\supseteq Y$ by \cite{NS2}, Lemma 3.5.

\bigskip

\noindent\textbf{Proof of Lemma \ref{bdedranklemma}.} $S$ is a simple group of
rank at most $r$ and $1\neq f\in\mathrm{Aut}(S)$. Then $C=\mathrm{C}_{S}(f)$
is a proper subgroup of $S$, and the main result of \cite{BCP} implies that
$\left\vert S:C\right\vert \geq\left\vert S\right\vert ^{\varepsilon(r)}$
where $\varepsilon(r)>0$ depends only on $r$. It follows that $\left\vert
[S,f]\right\vert =\left\vert S:C\right\vert \geq\left\vert S\right\vert
^{\varepsilon(r)}$.

\bigskip

\noindent\textbf{Proof of Lemma \ref{ranklemma}. }We are given a simple group $S$ and
$f\in Q(S)$. We have to show that if $\mathrm{rank}(S)$ is big enough then
there exists $s\in S$ such that $\log\left\vert [S,sf]\right\vert
<\varepsilon\log\left\vert S\right\vert $ (we will not distinguish between $s
$ and the inner automorphism it induces.)

If $S$ is a large alternating group, we can choose $s$ so that $sf$ is either
$1$ or (conjugation by) a transposition, and the claim is clear.

Otherwise, we may assume that $S$ is a classical group of dimension $n$ over a field of size $q$. Note that when $S$ is
an orthogonal group then $\mathrm{PGO}_{n}^{\epsilon}(q)/\mathrm{PSO}%
_{n}^{\epsilon}(q)$ is generated by a single reflection if $q$ is odd and is
trivial if $q$ is even (or $n$ is odd). At the same time $\mathrm{PSO}_{n}^{\epsilon
}(q)/\mathrm{Inn}(S)$ is generated by a product of two reflections if $q$ is
odd and is generated by a transvection if $q$ is even, see \cite{GLS} section
2.7($\epsilon\in\{\pm1\}$). In all cases, there exists $s\in S$ such that
$sf=h,$ where $h$ is an automorphism of $S$ such that $h$ is a product of at
most three reflections/transvections in case when $S$ is an orthogonal group
or else $h$ is a diagonal element with $n-1$ eigenvalues equal to 1 in case
$S$ is $\mathrm{PSU}_{n}(q)$ or $\mathrm{PSL}_{n}(q)$. In each case,
$\log|C_{S}(h)|/\log|S|\rightarrow1$ as $\mathrm{rank}(S)\rightarrow\infty$
uniformly in $q$. Lemma \ref{ranklemma} follows.

\bigskip

\noindent\textbf{Proof of Lemma \ref{notinQ(S)}. }Now $S$ is a simple group
and $f\in\mathrm{Aut}(S)\smallsetminus Q(S)$. This implies that $S$ is of Lie
type, and in view of Lemma \ref{bdedranklemma} we may assume that $S$ is a
classical group of large rank. We have to show that $\log\left\vert
\mathrm{C}_{S}(f)\right\vert /\log\left\vert S\right\vert \leq1-\varepsilon
_{\ast}$ where $\varepsilon_{\ast}>0$. \medskip

\textbf{Case I} Suppose that $f\in\mathrm{InnDiag}(S)$.$\medskip$

This can only happen when $S$ is an orthogonal group in even dimension or a
symplectic group and $q$ is odd. In both cases the description of
$\mathrm{Inndiag}(S)$ of \cite{GLS} for the classical groups in section 2.7
and the definition of $Q(S)$ show that $f$ is a similarity, i.e. there is some
$\lambda\in\mathbb{F}_{q}^{\ast}$ with $(f(v),f(v))
=\lambda (v,v) $ for all $v\in V$, the natural module for
$S$. Here $(.,.)$ is the natural bilinear form on $V$. Moreover $\lambda$ is not a square in $\mathbb{F}_{q}^{\ast}$, for if
$\lambda=\mu^{2}$ then $\mu^{-1}f=f\in\mathrm{PGO}_{2n}^{\pm}(q)=Q(S)$,
contradiction. We shall call $f$ a \emph{proper} similarity if $\lambda
\not \in (\mathbb{F}_{q}^{\ast})^{2}$. By considering an appropriate odd power
of $f$ we may assume that $f$ is a semisimple element of $\mathrm{GL}(V)$. Let
$t$ be the maximal multiplicity of some eigenvalue of $f$ over the algebraic
closure of $V$. We claim that $t\leq\dim V/2$. For if $t>\dim V/2$ then $t$
belongs to some rational eigenvalue $\mu\in\mathbb{F}_{q}^{\ast}$ and there is
some element $v$ of the $\mu$-eigenspace of $f$ with $(v,v)
\not =0$ (because the maximal dimension of totally isotropic subspaces of $V$
is $\dim V/2$). But then $\lambda (v,v) =(f(v),f(v)) =\mu^{2}(v,v)$, a contradiction since
$\lambda$ is non-square. This establishes the claim. Now Lemma 3.4 of
\cite{LS} (ii) shows that
\[
\left\vert f^{S}\right\vert >c^{\prime}q^{(\dim V-t)\dim V/2}>c^{\prime}q^{(\dim
V)^{2}/4}%
\]
for some constant $c^{\prime}>0$. This means that $\log|\mathrm{C}%
_{S}(f)|/\log|S|$ is bounded away from $1$. \medskip

\textbf{Case II} Suppose that $f \not \in \mathrm{InnDiag}(S)$. \medskip

\textbf{(a)} Assume first that

\begin{itemize}
\item either $S$ is untwisted and $f$ does not involve a graph automorphism or else

\item $S$ is twisted and $f$ has order at least $3$ modulo $\mathrm{InnDiag}%
(S)$.
\end{itemize}

We follow the methods of \cite{NS2}, Subsection \textbf{4.1.5} Let $L$ be the
adjoint simple algebraic group with a Steinberg morphism $F$ such that $S$ is
the socle of $L_{F}$. Then $L_{F}=\mathrm{InnDiag}(S)$ by \cite{GLS} Lemma
2.5.8 (a) and Theorem 2.2.6 (e). We may write $f$ as $f=\phi g$ where
$g\in\mathrm{Inndiag}(S)$ and $\phi$ is a field automorphism of order $m\geq2$
($m\geq3$ if $S$ is twisted). We have $F=\phi^{m}$ if $S$ is untwisted,
$F^{2}=\phi^{m}$ if $S$ is twisted. Since $L$ is connected, Lang's theorem
implies that there is some $g_{0}\in L$ such that $g=g_{0}^{\phi}g_{0}^{-1}$.
Let $x \in L$. The following pair of conditions on $x$ are equivalent: \medskip

(i) $x\in S$ and $x$ is fixed by $\phi g$, 

(ii) $y:=x^{g_{0}}$ is fixed  $\phi$ and $x=y^{{g_{0}}^{-1}%
}$ is fixed by $F$. \medskip

Ignoring the second part of (ii) we see that $|C_S(f)| \leq |L_\phi|$. Now $\log |L_\phi| \sim \dim L \log |K_\phi|$ where $K_\phi$ is the fixed field of $\phi$. On the other hand if $S=L_F$ is untwisted then $\log |S|\sim \dim L\log |K_F|$ where $K_F$ is the fixed field of $F=\phi^m$, while if $S$ is twisted then $\log|S|\sim \dim L \log| K_{F^2}|/2$, where $K_{F^2}$ is the fixed field of $F^2=\phi^m$. We conclude that $\log |L_\phi| \sim \frac{a}{m} \log |S| $ where $a=1$ if $S$ is untwisted and $a=2$ otherwise. In both cases $|C_S(F)| \leq |S|^{2/3}$.
\medskip

The remaining cases are

\begin{description}
\item[(b)] $S=A_{n}(q)$ and $f=g\phi\tau$ for an element $g\in\mathrm{GL}%
_{n+1}(q)$, field automorphism $\phi$ and the graph automorphism $\tau$.

\item[(c)] $S=\,^{2}\!A_{n}(q)$ and $f=g[q]$ for $g\in \mathrm{U}_{n+1}(q)$. Here $[q]$ denotes the automorphism of 
$\mathrm{U}_{n+1}$ induced by the field automorphism $x \mapsto x^q$.

\item[(d)] $S=D_{n}(q)$ and $f=g\phi\tau$ with $g\in\mathrm{Inndiag}(S)$,
where either the field automorphism $\phi$ is nontrivial or $g\in
\mathrm{InnDiag}(S)\smallsetminus\mathrm{PSO}_{n}(q)$ is a proper similarity.

\item[(e)] $S=\,^{2}\!D_{n}(q)$ with $f=g[q]$ where $g\in\mathrm{InnDiag}%
(S)\smallsetminus\mathrm{PSO}_{n}(q)$ is a proper similarity.
\end{description}

We consider these in turn.$\medskip$

\emph{Case (b).} Here $\tau$ is conjugate to the involution $x\mapsto (x^{T})^{-1}$
in $\mathrm{GL}_{n+1}$. Thus we may assume that $f(x)=g^{-1}(x^{T})^{-\phi}g$
for a matrix $g\in\mathrm{GL}_{n+1}$. Put $V=\mathbb{F}_{q}^{n+1}$, the
natural module for $S$ of column vectors. The condition $f(x)=x$ is equivalent
to $(x^{\phi})^{T}gx=g$, which is equivalent to the requirement that $x$
preserve the nondegenerate form $B:V\times V\rightarrow\mathbb{F}_{q}$ defined
by $B(v,w)=(v^{\phi})^{T}gw$. Now the result will follow from

\begin{lemma}
\label{space}Let $V$ be an $m$-dimensional vector space over a finite field
$F$, let $\phi$ be an automorphism of $F$, and let $B$ be a form on $V$ which
is nondegenerate and such that $B(\alpha v,\beta w)=\alpha^{\phi}\beta B(v,w)$
for any $\alpha,\beta\in F$. Then the subgroup $G_{B}$ of elements
$x\in\mathrm{GL}(V)$ that preserve $B$ has size at most $|F|^{m(m+1)/2}$.
\end{lemma}

\noindent We omit the proof, an exercise in linear algebra. Note that Lemma
\ref{space} estimates the fixed points of $f$ in $\mathrm{SL}_{n+1}(q)$ which
is the universal cover $\tilde{S}$ of $S$. The trivial bound $\left\vert
[S,f]\right\vert \geq|[\tilde{S},f]||\mathrm{Z}(\tilde{S})|^{-1}$ together with
$\log|\mathrm{Z}(\tilde{S})|/\log|S|\rightarrow0$ as $\mathrm{rank}(S)\rightarrow\infty$ then
completes the proof of case (b).$\medskip$

\emph{Case (c). }The argument here also follows the idea in \cite{NS2},
Subsection 4.1.5. Let $X$ be the algebraic group $\mathrm{GL}_{n+1}$. Let
$[q]$ be the morphism $x\mapsto x^{[q]}$ of $X$ and let $F$ be the Frobenius
morphism $x\mapsto x^{\tau\lbrack q]}$ whose fixed points on $\mathrm{SL}%
_{n+1}$ is the universal cover $\tilde{S}=\mathrm{SU}_{n+1}(q)$ of $S$. By
Lang's theorem we can write $g=g_{0}^{[q]}g_{0}^{-1}$ for some $g_{0}\in X$.
Since $g$ is fixed by $F$, the element $h:=g_{0}^{-F}g_{0}$ is fixed by $[q]$,
i.e. $h\in\mathrm{GL}_{n+1}(q)$. Let $x \in \mathrm{SL}_{n+1}$. The following two conditions on $x$ are equivalent: \medskip 
\medskip

(i) $x\in \tilde S$ and $x$ is fixed by $[q]g$, \medskip

(ii) $y:=x^{g_{0}} \in \mathrm{SL}_{n+1}$ is fixed by $[q]$ and $x=y^{{g_{0}}^{-1}%
}$ is fixed by $F$. \medskip

 Therefore the
number of fixed points of $[q]g$ on $\tilde{S}$ is equal to the number of
elements $y\in\mathrm{SL}_{n+1}(q)$ such that
\[
(y^{g_{0}^{-1}})^{F}=x^{F}=x=(y^{F})^{g_{0}^{-F}}=y^{g_{0}^{-1}},
\]
equivalently%
\[
(y^{F})^{(g_{0}^{-F}g_{0})}=y.
\]
Observing that for $y\in\mathrm{SL}_{n+1}(q)$ we have $y^{F}=y^{\tau}$ the
last condition becomes $y^{\tau h}=y$. By Case (a) the number of such $y$ is
at most $q^{(n+1)n/2}$, which is about $|\tilde{S}|^{1/2}$. We have proved
that $|[\tilde{S},f]|\geq|\tilde{S}|^{1/2-o(1)}$; the corresponding result for
$\left\vert [S,f]\right\vert $ follows just as in case (b).

$\medskip$

\emph{Case (d).} Let $U$ be the fixed point subgroup of the graph automorphism
$\tau$ in $S=D_{n}(q)$. Then $U=B_{n-1}(q)$ and $\log\left\vert U\right\vert
/\log\left\vert S\right\vert \rightarrow1$ as $n\rightarrow\infty$. Hence it
is enough to prove that $\log\left\vert \mathrm{C}_{S}(f)\cap U\right\vert
/\log\left\vert S\right\vert \leq1-\varepsilon$ for some fixed $\varepsilon
>0$. Now $\mathrm{C}_{S}(f)\cap U$ is contained in the fixed point set of
$g\phi$ on $S$, which has size at most $|S|^{1-\epsilon}$ by Case II (a) if
$\phi\not =1$ and by Case I otherwise.$\medskip$

\emph{Case (e).} This is similar to Case (d) on putting $U=\mathrm{C}%
_{S}([q])\cong B_{n-1}(q)$, noting that $\log\left\vert U\right\vert
/\log\left\vert S\right\vert \rightarrow1$ as $n\rightarrow\infty,$ and
applying Case I to $\mathrm{C}_{S}(g)$.

\subsection{The general case}

Assume now that $G$ is a compact group with $G/G^{0}$ finitely generated. We
will show that $G$ has a proper dense normal subgroup -- say $G$ has DNS -- if
and only if one of the following holds:

\begin{description}
\item[(a)] $G^{\mathrm{ab}}$ is infinite

\item[(b)] $G$ has a strictly infinite semisimple quotient

\item[(c)] $G$ has \emph{Q}-almost-simple quotients of unbounded ranks.
\end{description}

Suppose that $G$ is infinite and abelian. Then either $G$ maps onto an
infinite finitely generated abelian profinite group, or $G^{0}$ has finite
index in $G$. In the first case, $G$ has DNS by a remark in subsection
\ref{profsec}; in the second case, $G$ maps onto a non-trivial torus
(\cite{HM}, Proposition 8.15), hence onto $\mathbb{R}/\mathbb{Z}$: then the
inverse image of $\mathbb{Q}/\mathbb{Z}$ is a dense proper subgroup of $G$.
Thus in general, if (a) holds then $G$ has DNS. If (b) holds, a dense normal
subgroup is provided by the inverse image in $G$ of the restricted direct
product of simple factors in a strictly infinite semisimple quotient of $G$.

Suppose that (c) holds but neither (a) nor (b) does. If $G$ has \emph{Q}%
-almost-simple \emph{finite} quotients of unbounded ranks then so does
$G/G^{0}$, and then $G$ has DNS by subsection \ref{profsec}. Otherwise, there
exists a strictly increasing sequence $(n_{i})$ (with $n_{1}\geq3)$ such that
$G$ maps onto each $\mathrm{Aut}(S_{i})$ where $S_{i}\cong\mathrm{PSO}%
(2n_{i})$ for each $i$. Then for each $i$, the inverse image $D_{i}$ in $G$ of
$\mathrm{Inn}(S_{i})$ has index $2.$ Since $G$ has only finitely many open
subgroups of index $2$, we can replace $(n_{i})$ with an infinite subsequence
and reduce to the case where $D_{i}=D$ is constant. Then $G$ has closed normal
subgroups $N_{i}<D$ such that $D/N_{i}\cong S_{i}$ and $G/N_{i} $ induces
$\mathrm{Aut}(D/N_{i})$ on $D/N_{i}$; replacing $G$ by a quotient we may
assume that $\bigcap_{i=1}^{\infty}N_{i}=1$.

Then $D=\prod_{i}S_{i}$ where $S_{i}=\bigcap_{j\neq i}N_{j}\cong%
\mathrm{PSO}(2n_{i})$. Also $G=D\left\langle y\right\rangle $ where $y^{2}\in
D$ and $y$ acts on $S_{i}$ like $s_{i}\tau_{i},$ where $s_{i}\in
S_{i}=\mathrm{Inn}(S_{i})$ and $\tau_{i}$ is the non-trivial graph
automorphism of $S_{i}$ given by conjugation by the diagonal matrix
$\mathrm{diag}(1,1, \cdots,1,-1) \in\mathrm{O}(2n)$.

Let $\tau=(s_{i}^{-1})_{i\in\mathbb{N}}\cdot y\in Dy$. Then $\tau$ induces
$\tau_{i}$ on each $S_{i}$ and $\tau^{2}\in\mathrm{C}_{G}(D)=1.$ Put
$N=\left\langle \tau^{G}\right\rangle $. Then for each $i$ we have $S_{i}\cap
N\supseteq\lbrack S_{i},\tau]\neq1$, so $S_{i}\cap\overline{N}$ is a
non-trivial closed normal subgroup of $S_{i}$ and so $S_{i}\leq\overline{N}$.
Therefore $D\leq\overline{N}$ and it follows that $\overline{N}=G$.

We claim that $N\neq G$. To see this, observe that $\dim(S_{i})=2n^{2}-n$
while $\dim\left(  \mathrm{C}_{S_{i}}(\tau_{i})\right)  =\dim\left(
\mathrm{O}(2n-1)\right)  =2(n-1)^{2}-n+1$; therefore%

\[
\frac{\dim[S_{i},\tau_{i}]}{\dim(S_{i})}\rightarrow0\text{ as }i\rightarrow
\infty.
\]
This implies that for each $m$ there exist $i(m)\in\mathbb{N}$ and
$s_{i(m)}\in S_{i(m)}\smallsetminus\lbrack S_{i(m)},\tau]^{\ast m}$, and we
may choose $i(m)>i(m-1)$ for each $m>1$. Now let $h\in D$ be such that
$h_{i(m)}=s_{i(m)}$ for all $m$. We claim that $h\notin N$.

Indeed, suppose that%
\[
w=\prod_{j=1}^{m}\tau^{g_{j}}\in D,
\]
where without loss of generality $g_{j}=(g_{j,i})_{i}\in D$. Then $m$ is even
and%
\[
w_{i}=[x_{1},\tau][x_{2}^{\tau},\tau]\ldots\lbrack x_{m-1},\tau][x_{m}^{\tau
},\tau]\in\lbrack S_{i},\tau]^{\ast m}%
\]
where $x_{j}=g_{j,i}$. Therefore $w_{i(m)}\neq h_{i(m)},$ so $h\neq w$.

Thus $N$ is a proper dense normal subgroup of $G$.

\bigskip

For the converse, let $N$ be a proper dense normal subgroup of $G$, and assume
that neither (a) nor (b) holds; we will show that (c) must hold.

$\medskip$

1. If $G^{0}N<G$ we are done by the profinite case. So we may assume that
$G^{0}N=G$.

$\medskip$

2. Let $Z=\mathrm{Z}(G^{0})$. Suppose that $NZ=G$. Then $G^{\prime}\leq N$;
but we have assumed that $G^{\mathrm{ab}}$ is finite, so $N$ has finite index
in $G,$ hence contains $G^{0}$, whence $N=G$, a contradiction. Therefore
$NZ<G$, and replacing $N$ by $NZ$ we may assume that $Z\leq N$. Now replacing
$G$ by $G/Z$ we may suppose that $Z=1$. In this case we have%
\[
G^{0}=\prod_{i\in I}S_{i}%
\]
where each $S_{i}$ is a compact connected simple (and centreless) Lie group.

$\medskip$

3. Put $D=G^{0}\cap N$. Then $[G^{0},N]\leq D<G^{0}$. It follows that%
\[
G^{0}=G^{0\prime}\leq\lbrack G^{0},G]=[G^{0},\overline{N}]\leq\overline{D},
\]
so $D$ is dense in $G^{0}$. Suppose that $S_{i}\nleqq D$ for some $i$. As
$S_{i}$ is abstractly simple, we have $S_{i}\cap D=1$, whence $D\leq
\mathrm{C}_{G^{0}}(S_{i})$, a proper closed subgroup of $G^{0}$. It follows
that $S_{i}\leq D$ for every $i$. Since $D<\overline{D}$ this implies that the
index set $I$ must be infinite.

Since $G/G^{0}$ is finitely generated, we have $G=G^{0}\overline{\left\langle
y_{1},\ldots,y_{d}\right\rangle }$ for some $y_{l}\in N$. Then $[G^{0}%
,y_{l}]\subseteq D$ for each $l$. Applying \cite{NS2}, Proposition 5.18 we
deduce that there exists an infinite subset $J$ of $I$ such that each $y_{l}$
normalizes $S_{i}$ for every $i\in J$. As $\mathrm{N}_{G}(S_{i}) $ is closed
and contains $G^{0}$, it follows that $S_{i}$ is normal in $G$ for every $i\in
J$. We may take $J=\left\{  i\in I\mid S_{i}\vartriangleleft G\right\}  $.

Put $P=\prod_{i\in J}S_{i}$ and $C=\mathrm{C}_{G}(P)$. Suppose that $CN=G$.
Then%
\[
P=P^{\prime}\leq\lbrack P,G]=[P,N]\leq N\text{.}%
\]
Now apply the preceding argument to $G/P$: since $(G/P)^{0}=(\prod_{i\in
I\smallsetminus J}S_{i})P/P$, this implies that $G$ normalizes $S_{j}$ for
infinitely many $j\in I\smallsetminus J$, which contradicts the definition of
$J$.

It follows that $CN<G$. So replacing $N$ by $CN$ and then replacing $G$ by
$G/C$, we may assume that $C=1$. As $C\cap G^{0}=\prod_{i\in I\smallsetminus
J}S_{i}$, this means in particular that $S_{i}\vartriangleleft G$ for all
$i\in I$ (renaming $J$ to $I$), and that $\mathrm{C}_{G}(G^{0})=1$.

Now $\mathrm{Out}(S_{i})$ embeds in $\mathrm{Sym}(3)$ for each $i$. As
$G/G^{0}$ is a finitely generated profinite group, it admits only finitely
many continuous homomorphisms to $\mathrm{Sym}(3)$, so $G$ has an open normal
subgroup $H\geq G^{0}$ such that $H$ induces inner automorphisms on each
$S_{i}$. Then $H=\mathrm{C}_{H}(G^{0})G^{0}=G^{0}$; thus $G/G^{0}$ is finite.
Hence there exists a finite subset $Y$ of $N$ such that $G=G^{0}Y$.

$\medskip$

4. Put $C_{i}=\mathrm{C}_{G}(S_{i})$ and now set $J=\left\{  i\in I\mid
C_{i}S_{i}=G\right\}  $. For $i\in J$ put $K_{i}=\bigcap_{i\neq j\in J}C_{j}$
and $X=\bigcap_{i\in J}C_{i}$. Then for $i\in J$ we have%
\[
S_{i}\cong XS_{i}/X\vartriangleleft K_{i}/X\cong G/C_{i}\cong S_{i},
\]
so $K_{i}=X\times S_{i}$. Hence the image of $G/X$ in the product $\prod_{i\in
J}G/C_{i}\cong\prod_{i\in J}S_{i}$ contains the restricted direct product, and
it follows that $G/X\cong\prod_{i\in J}S_{i}$. Since we have assumed that $G$
has no strictly infinite semisimple quotient, it follows that $J$ is finite.

As $S_{i}\leq D$ for each $i$, we may now replace $G$ by $G/\prod_{i\in
J}S_{i}$, and so assume that $J$ is empty. Then for each $i$ there exists
$y(i)\in Y$ such that $y(i)$ induces an outer automorphism on $S_{i}$.

$\medskip$

5. In the next subsection we will prove

\begin{proposition}
\label{outer}Let $S$ be a compact connected simple and centreless Lie group
and $y$ an outer automorphism of $S$. Then%
\[
S=([S,y]\cdot\lbrack S,y^{-1}])^{\ast k}%
\]
where%
\begin{align*}
k  &  =k_{0}\text{ if }S\ncong\mathrm{PSO}(2n)~~\forall n,\\
k  &  \leq k(n)\text{ if }S\cong\mathrm{PSO}(2n),
\end{align*}
$k_{0}\in\mathbb{N}$ is an absolute constant and $k(n)\in\mathbb{N}$ depends
on $n$.
\end{proposition}

Now let $t\geq3$ and let
\[
J(t)=\left\{  i\in I\mid\text{ }S_{i}\ncong\mathrm{PSO}(2n)~\forall
n>t\right\}  .
\]
Put $k(t)=\max\{k_{0},~k(n)~(n\leq t)\}$. Then for each $i\in J(t)$ we have%
\[
S_{i}=\prod_{y\in Y}([S_{i},y]\cdot\lbrack S_{i},y^{-1}])^{\ast k(t)}\text{.}%
\]
Therefore%
\[
\prod_{i\in J(t)}S_{i}=\prod_{y\in Y}\left(  \left[  \prod\nolimits_{i\in
J(t)}S_{i},y\right]  \cdot\left[  \prod\nolimits_{i\in J(t)}S_{i}%
,y^{-1}\right]  \right)  ^{\ast k(t)}\subseteq N\text{,}%
\]
since the product in the middle is a closed set. As $N<G=G^{0}N$ it follows
that $\prod_{i\in J(t)}S_{i}<G^{0}$; thus $J(t)\neq I$.

Thus for every $t$ there exist $n>t$ and $i\in I$ such that $S_{i}%
\cong\mathrm{PSO}(2n)$, and then $G$ has the \emph{Q}-almost simple quotient
$G/C_{i}\cong\mathrm{Aut}(\mathrm{PSO}(2n))$. Thus (c) holds.

\subsection{More lemmas}

Throughout this section, we take $S$ to be a compact connected simple Lie
group (see for example \cite{H}, Table IV on page 516). We assume that $S$ has
an outer automorphism. Such an automorphism is the product of an inner
automorphism and a non-trivial graph automorphism; this only exists when $S$
has type $A_{n},D_{n},$ or $E_{6}$. We choose and fix a maximal torus $T$ of
$S$, a root system $\Phi$ of characters of $T$, and a set of fundamental roots
$\{\alpha_{1},\ldots,\alpha_{r}\}$. Throughout, $r=r(S)$ will denote the rank
of $S$, and $W$ the Weyl group.

We need not assume that $\mathrm{Z}(S)=1$, but will sometimes for brevity
identify elements of $S$ with the corresponding inner automorphisms.

The function $\lambda:T\rightarrow\lbrack0,1]$ was defined in \cite{NS2},
Subsection 5.5.4:%
\[
\lambda(t)=(\pi r)^{-1}\sum_{i=1}^{r}|l(\alpha_{i}(t))|
\]
where $l(e^{i\theta})=\theta$ for $\theta\in(-\pi,\pi]$.

Proposition \ref{outer} depends on Lemma 5.19 of \cite{NS2} which we restate
here in the following form:

\begin{proposition}
\label{lamb} For each $\epsilon>0$ there is an integer $k=k^{\prime}%
(\epsilon)$ such that if $g\in T$ satisfies $\lambda(g)>\epsilon$ then
$S=(g^{S}\cup g^{-S})^{\ast k}$.
\end{proposition}

We shall prove

\begin{lemma}
\label{brank} \emph{(i) }There exists $\epsilon=\epsilon(S)>0$ such that: for
each $f\in\mathrm{Aut}(S)\smallsetminus\mathrm{Inn}(S)$ there exist an element
$g\in\lbrack S,f]$ and a conjugate $g_{1}$ of $g$ with $g_{1}\in T$ and
$\lambda(g_{1})>\epsilon$.\newline\emph{(ii)} If $S=(\mathrm{P})\mathrm{SU}%
(n)$ where $n>30$ we can take $\epsilon(S)=(200\pi)^{-1}$.
\end{lemma}

A simple calcuation shows that if $g\in\lbrack S,f]$ then $g^{S}%
\subseteq\lbrack S,f][S,f^{-1}]$ and $g^{-S}\subseteq\lbrack S,f^{-1}][S,f]$,
and so%
\[
(g^{S}\cup g^{-S})^{\ast k}\subseteq\left(  \lbrack S,f][S,f^{-1}%
][S,f][S,f^{-1}]\right)  ^{\ast k}=\left(  [S,f][S,f^{-1}]\right)  ^{\ast2k}.
\]

Now if $r>6$, then $S$ has type $D_{r}$ or $A_{r}$. Among centreless compact
simple groups, the one of type $D_{r}$ is $\mathrm{PSO}(2r)$ and the one of
type $A_{r}$ is $\mathrm{PSU}(r+1)$. So Proposition \ref{outer} will follow
from these results on setting%
\begin{align*}
k_{1}  &  =\max\left\{  k^{\prime}(\epsilon(S))\mid r(S)\leq30\right\}  ,\\
k_{0}  &  =2\max\left\{  k_{1},~k^{\prime}((200\pi)^{-1})\right\}  ,\\
k(n)  &  =2k^{\prime}(\epsilon(\mathrm{PSO}(2n)))
\end{align*}
(this makes sense because only finitely many groups $S$ have rank at most $30
$).

\bigskip

\textbf{Proof of Lemma \ref{brank}(i). }For $g\in S$, choose an $S$-conjugate
$g^{\prime}$ of $g$ inside $T$ and set%
\[
\eta(g)=\max\{l(\alpha(g^{\prime}))\ |\ \alpha\in\Phi\};
\]
as different choices for $g^{\prime}$ lie in the same orbit of $W$ on $T,$
this definition is independent of the choice of $g^{\prime}$. It is clear that
$\eta(g)=0$ is equivalent to $g\in\mathrm{Z}(S)$.

Now given $g\in S\backslash \mathrm{Z}(S)$, we have $\eta(g)=l(\alpha(g^{\prime}))>0$ for some some
root $\alpha\in\Phi$, and there exists $w\in W$ with $\alpha^{w}=\alpha_{j}$
for some $j$. Setting $g_{1}=g^{\prime w}$ we see that%
\[
\lambda(g_{1})\geq(\pi r)^{-1}l(\alpha_{j}(g_{1}))=(\pi r)^{-1}l(\alpha
(g^{\prime}))=(\pi r)^{-1}\eta(g).
\]

Suppose now that the statement (i) is false. Then we can find a sequence
$f_{i}=g_{i}s_{i}\in\mathrm{Aut}(S),$ with $g_{i}\in S$ and $s_{i}$ a
nontrivial graph automorphism, such that
\begin{equation}
\sup\{\eta(g)\mid g\in\lbrack S,f_{i}]\}\rightarrow0\text{ as }i\rightarrow
\infty. \label{limit0}%
\end{equation}
Since $S$ is a compact group and $\mathrm{Out}(S)$ is finite we can find a
subsequence $f_{i(j)}=g_{i(j)}s_{i(j)}$ with $s_{i(j)}=s$ the same nontrivial
graph automorphism for all $j\geq1$ and $(g_{i(j)})_{j}$ converging to an
element $g\in S$. Thus the subsequence $f_{i(j)}$ converges to the
automorphism $f_{\infty}=gs$ in $\mathrm{Aut}(S)$. Now (\ref{limit0}) implies
that $[S,f_{\infty}]\subseteq\mathrm{Z}(S)$ and hence that $f_{\infty}=1$
since $S=[S,S]$, a contradiction since $s=g^{-1}f_{\infty}$ is not inner.

\bigskip

For Lemma \ref{brank}(ii), we fix $S=\mathrm{SU}(n)$ with $n>30$ and choose
$T$ to be the group of diagonal matrices $A=\mathrm{diag}(x_{1},\ldots,x_{n})$
in $S$; then $\Phi$ consists of all characters $\alpha_{i,j}$ defined by
$\alpha_{i,j}(A)=x_{i}x_{j}^{-1}$. The Weyl group $W$ of $S$ is $\mathrm{Sym}%
(n)$ acting on $T$ by permuting the eigenvalues. The function $\lambda
:T\rightarrow\lbrack0,1]$ is given by%
\[
\lambda(A)=(\pi(n-1))^{-1}\sum_{i=1}^{n-1}\left\vert l(x_{i}x_{i+1}%
^{-1})\right\vert .
\]

The only nontrivial graph automorphism of $\mathrm{SU}(n)$ is induced by
complex conjugation of the matrix entries.

\begin{lemma}
\label{scalar} Suppose that $A\in T$ satisfies $\lambda(A^{w})\leq\epsilon$
for every $w\in W$. Let $x_{1},\ldots,x_{n}$ be all the eigenvalues of $A$
listed with multiplicities. Then there exists an eigenvalue $x$ of $A$ such
that $\left\vert l(xx_{i}^{-1})\right\vert <20\pi\epsilon$ for at least
$9n/10$ values of $i\in\{1,\ldots,n\}$.
\end{lemma}

\begin{proof}
Suppose the claim is false. Then for any eigenvalue $y$ of $A$, at least one
tenth of the other eigenvalues $x$ satisfy $\left\vert l(xx_{i}^{-1}%
)\right\vert \geq20\pi\epsilon$. Hence we can reorder $x_{1},\ldots,x_{n}$ as
$y_{1},\ldots,y_{n}$ so that $\left\vert l(y_{i}y_{i+1}^{-1})\right\vert
\geq20\pi\epsilon$ for each $i=1,\ldots,[n/10]$. This means that
\[
\lambda(\mathrm{diag}(y_{1},\ldots,y_{n}))\geq(\pi(n-1))^{-1}\times
20\pi\epsilon\times\lbrack n/10]>\epsilon,
\]
contradicting the hypothesis.
\end{proof}

\medskip

\textbf{Proof of Lemma \ref{brank}(ii).} Consider an element $a\in
\mathrm{SU}(n)$ which has eigenvalues $1,\omega:=\exp(\pi i/3)$ and $-1$, each
with multiplicity $m:=[(n-1)/3]$. Then $a^{f}$ has eigenvalues $1,\omega^{-1}$
and $-1$ with the same multiplicity $m$.

We claim that the element $b:=a^{-1}a^{f}$ has an $S$-conjugate $b_{1}\in T$
with $\lambda(b_{1})>(200\pi)^{-1}$. Suppose this is not true. Then by Lemma
\ref{scalar} there exist a $b$-invariant subspace $V_{1}$ of $\mathbb{C}^{n}$
with $\dim V_{1}\geq9n/10$ and a complex number $\sigma$ such that all
eigenvalues of $b$ on $V_{1}$ are of the form $\sigma\exp(it)$ with
$\left\vert t\right\vert <1/10$. As a consequence, for any unit vector $v\in
V_{1}$ we have $|b\cdot v-\sigma v|<|1-\exp(i/10)|<1/10$ (because all
eigenvalues of $b-\sigma\mathrm{Id}$ on $V_{1}$ have norm at most
$|1-\exp(i/10)|$).

Whatever the complex number $\sigma$, there is some $\mu\in\{1,\omega,-1\}$
such that $l(\mu\sigma)\in\lbrack\pi/6,5\pi/6]\cup\lbrack-\pi/2,-5\pi/6]$.
This means that $l(\sigma\mu x^{-1})\geq\pi/6$ for each $x=1,-1,\omega^{-1}$.
Therefore $\left\vert \sigma\mu-x\right\vert >1/2$ for each such $x$.

Let $V_{2}$ be the $\mu$-eigenspace of $a$ and let $V_{3}$ be the sum of the
$1$,$-1$ and $\omega^{-1}$-eigenspaces of $a^{f}$. We have $\dim
V_{3}=3m,~\dim V_{2}=m$, while $n\leq3m+3$ and $\dim V_{1}\geq9n/10$. As
$n>30$ we have%
\[
\dim V_{1}+\dim V_{2}+\dim V_{3}>2n,
\]
which implies that $V_{1}\cap V_{2}\cap V_{3}$ is nonempty. Pick a unit vector
$v\in V_{1}\cap V_{2}\cap V_{3}$. Since $b=a^{-1}a^{f}$ we have $ab\cdot
v=a^{f}\cdot v$. We can write $bv=\sigma v+u$ where $u$ is a vector of norm
less than $1/10$. Since $v$ and $\sigma v$ belong to $V_{2}$ we have $ab\cdot
v=\mu\sigma v+u_{1}$ where $u_{1}=au$ has norm less than $1/10$. On the other
hand $v\in V_{3}$ and so we may write $v=w_{1}+w_{2}+w_{3}$ where $w_{1}%
,w_{2},w_{3}$ are $x_{i}$-eigenvectors of $a^{f}$ where $(x_{1},x_{2}%
,x_{3})=(1,\omega^{-1},-1)$. But distinct eigenspaces of a unitary operator
are mutually orthogonal, hence $|w_{1}|^{2}+|w_{2}|^{2}+|w_{3}|^{2}=|v|^{2}%
=1$. Now
\[
\mu\sigma(w_{1}+w_{2}+w_{3})+u_{1}=\mu\sigma v+u_{1}=ab\cdot v=a^{f}\cdot
v=x_{1}w_{1}+x_{2}w_{2}+x_{3}w_{3},
\]
giving%
\[
\sum_{i=1}^{3}(\mu\sigma-x_{i})w_{i}=-u_{1},
\]
and since $\left\vert \mu\sigma-x_{i}\right\vert \geq1/2$ for each $i=1,2,3$
by the choice of $\mu$ this implies that
\[
10^{-2}>\left\vert u_{1}\right\vert ^{2}=\sum_{i=1}^{3}\left\vert \mu
\sigma-x_{i}\right\vert ^{2}\left\vert w_{i}\right\vert ^{2}\geq\frac{1}%
{4}\sum_{i=1}^{3}\left\vert w_{i}\right\vert ^{2}=1/4.
\]
This contradiction completes the proof.


\begin{thebibliography}{9999}                                                                                             %


\bibitem[BCP]{BCP}Babai, L., Cameron, P. J. and P\'{a}lfy, P.: On the orders
of primitive groups with restricted non-abelian composition factors, \emph{J.
Algebra }\textbf{79} (1982), 161-168.

\bibitem[G]{G}Gorenstein, D.: \emph{Finite simple groups, }Plenum Press, New
York and London, 1982.

\bibitem[GLS]{GLS}Gorenstein, D., Lyons, R. and Solomon, R.: \emph{The
classification of the finite simple groups, no.3}, American Math. Soc.,
Providence, Rhode Island, 1998.

\bibitem[H]{H}Helgason, H.: \emph{Differential geometry, Lie groups and
symmetric spaces}, Graduate studies in Mathematics \textbf{34}, American Math.
Soc., Providence, Rhode Island, 2001.

\bibitem[HM]{HM}Hofmann, K. H. and Morris, S. A.: \emph{The structure of
compact groups}. 2nd edn., de Gruyter Studies in Mathematics, \textbf{25}.
Walter de Gruyter \& Co., Berlin, 2006.

\bibitem[LiSh]{LiSh}Liebeck, M. W. and Shalev, A.: Diameters of finite simple
groups: sharp bounds and applications, \emph{Annals of Math.} \textbf{154}
(2001), 383-406.

\bibitem[LS]{LS}Liebeck, M. W. and Shalev, A.: Simple groups, permutation
groups and probability, \emph{Journal of the Amer. Math. Soc.} \textbf{12},
(1999) 497-520.

\bibitem[NS1]{NS1}Nikolov, N. and Segal, D.: On finitely generated profinite
groups, I: strong completeness and uniform bounds, \emph{Annals of Math.
}\textbf{165} (2007), 171--238.

\bibitem[NS2]{NS2}Nikolov, N. and Segal, D.: Generators and commutators in
finite groups; abstract quotients of compact groups, \emph{Invent. Math.}, to appear.

\bibitem[W]{W}Segal, D.: \emph{Words: notes on verbal width in groups, }London
Math. Soc. Lecture Notes Series \textbf{361}, Cambridge Univ. Press,
Cambridge, 2009.
\end{thebibliography}
\end{document}